\documentclass[11pt]{amsart}
\usepackage{amsmath,amsthm,amsfonts,amssymb}%showlabels}%,mathrsfs}
\newtheorem{thm}{Theorem}[section]
\newtheorem{prop}[thm]{Proposition}
\newtheorem{cor}[thm]{Corollary}
\newtheorem{lem}[thm]{Lemma}
\newtheorem{prob}[thm]{Problem}
\theoremstyle{definition}
\newtheorem{defn}[thm]{Definition}
\newtheorem{rem}[thm]{Remark}

\newcommand{\IZ}{\mathbb Z}
\newcommand{\IR}{\mathbb R}

\newcommand{\C}{\mathcal C}
\newcommand{\U}{\mathcal U}
\newcommand{\V}{\mathcal V}
\newcommand{\W}{\mathcal W}

\newcommand{\WAP}{\mathrm{WAP}}
\newcommand{\AP}{\mathrm{AP}}
\newcommand{\SAP}{\mathrm{SAP}}
\newcommand{\FA}{\mathrm{FA}}
\newcommand{\FM}{\mathrm{FM}}

\newcommand{\cc}{\mathfrak c}
\newcommand{\w}{\omega}
\newcommand{\cl}{\mathrm{cl}}
\newcommand{\id}{\mathrm{id}}
\newcommand{\Ra}{\Rightarrow}

\title[Bicyclic semigroup in countably compact semigroups]{Embedding the bicyclic semigroup into countably compact topological semigroups}
\author{Taras Banakh}
\address{Instytut Matematyki, Uniwersytet Humanistyczno-Przyrodniczy Jana Kochanowskiego w Kielcach, Poland \\ and Department of
Mathematics, Lviv National University,  Universytetska 1, 79000,
Ukraine} \email{T.O.Banakh@gmail.com}

\author{Svetlana Dimitrova}
\address{National Technical University ``Kharkiv Polytechnical Institute", Frunze 21, Kharkiv, 61002, Ukraine}
\email{s.dimitrova@mail.ru}

\author{Oleg~Gutik}
\address{Department of Mechanics and Mathematics, Ivan Franko Lviv
National University, Universytetska 1, Lviv, 79000, Ukraine}
\email{ovgutik@yahoo.com}

\begin{document}
\begin{abstract} We study algebraic and topological properties
of topological semigroups containing a copy of the bicyclic
semigroup $\C(p,q)$. We prove that a  topological semigroup $S$
with pseudocompact square contains no dense copy of $\C(p,q)$. On
the other hand, we construct a (consistent) example of a  pseudocompact
(countably compact) Tychonoff semigroup containing a copy of $\C(p,q)$.
\end{abstract}



\maketitle

In this paper we study the structural properties of topological
semigroups that contain a copy of the bicyclic semigroup $\C(p,q)$
and present a (consistent) example of a Tychonoff pseudocompact (countably compact)
semigroup $S$ that contains $\C(p,q)$. This example shows that the
theorem of Koch and Wallace \cite{KW} saying that compact
topological semigroups do not contain bicyclic subsemigroups
cannot be generalized to the class of pseudocompact or countably compact
topological semigroups. Also this example shows that the presence of an inversion is essential in a result of Gutik and Repov\v s \cite{GR} who proved that the bicyclic semigroup does not embed into a countably compact topological inverse semigroup.

The presence or absence of a bicyclic subsemigroup in a given
(topological) semigroup $S$ has important implications for
understanding the algebraic (and topological) structure of $S$.
For example, the well-known Andersen Theorem \cite{Andersen1952}, \cite[2.54]{CP} 
 says that a simple semigroup with an
idempotent but without a copy of $\C(p,q)$ is completely simple
and hence by the Rees-Suschkewitsch Theorem \cite{Rees}, has the
structure of a sandwich product $[X,H,Y]_\sigma$ of two sets $X,Y$
and a group $H$ connected by a suitable sandwich function
$\sigma:Y\times X\to H$. The Rees-Suschkewitsch Theorem has also a topological version, see \cite{BDG}. 

Having in mind the mentioned result of Koch and Wallace \cite{KW},
I.I.~Guran asked if the bicyclic semigroup can be embedded into a
countably compact topological semigroup. In this paper we shall
find many conditions on a topological semigroup $S$ which forbid
$S$ to contain a bicyclic subsemigroup. One of the simplest
conditions is the countable compactness of the square $S\times S$.
On the other hand, we construct a Tychonoff pseudocompact semigroup that contains a bicyclic semigroup. Moreover, assuming the existence of a countably compact abelian torsion-free topological group without convergent sequences we shall construct
an example of a Tychonoff countably compact topological semigroup
that contains a copy of the bicyclic semigroup.

To construct such pathological semigroups, we shall study the operation of
attaching a discrete semigroup $D$ to a topological semigroup $X$
along a homomorphism $\pi:D\to X$. This construction has two
ingredients: topological and algebraic, discussed in the next four
sections. In section~\ref{s5} we establish some structure
properties of topological semigroups that contain a copy of the
bicyclic subsemigroup and in Section~\ref{s6}  we  construct our
main counterexample. Our method of constructing this
counterexample is rather standard and exploits the ideas of
D.~Robbie, S.~Svetlichny \cite{RS} (who constructed a countably
compact cancellative semigroup under CH) and A.~Tomita \cite{Tom96}
(who weakened the Continuum Hypothesis in their result to a weaker
version of Martin's Axiom).

All topological spaces appearing in this paper are assumed to be
Hausdorff.

\section{Attaching a discrete space to a topological space}

In this section we describe a simple construction of attaching a
discrete space $D$  to a topological space $M$ along a map
$\pi:D\to M$ and will investigate topological properties of the
obtained space $D\cup_\pi M$. Although all non-trivial applications concern infinite $D$, we do not restrict ourselves by infinite spaces and formulate our results for any (not necessarily infinite) discrete space $D$.

Let $D$ be a discrete topological space. If $D$ is infinite, then let $\alpha D=D\cup\{\infty\}$ be the Aleksandrov compactification of $D$. If $D$ is finite, then let $\alpha D=D\cup\{\infty\}$ be the topological sum of $D$ and the singleton $\{\infty\}$ for some point $\infty\notin D$. 

Given a map
$\pi:D\to M$ to a $T_1$-topological space $M$, consider the
closed subspace
 $$
D\cup_\pi M=\{(x,\pi(x)):x\in D\}\cup (\{\infty\}\times M)
 $$
of the product $\alpha D\times M$. We shall identify the space $D$
with the open discrete subspace $\{(x,\pi(x)):x\in D\}$ and $M$
with the closed subspace $\{\infty\}\times M$ of $D\cup_\pi M$.
Let $\bar \pi=\pi\cup\id_M:D\cup_\pi M\to M$ denote the projection
to the second factor. Observe that the topology of the space
$D\cup_\pi M$ is the weakest $T_1$-topology that induces the original
topologies on the subspaces $D$ and $M$ of $D\cup_\pi M$ and makes
the map $\bar\pi$ continuous.

The following (almost trivial) propositions describe some elementary properties of the space $D\cup_\pi M$.

\begin{prop}\label{p1.1}
If for some $i\le 3\frac12$ the space $M$ satisfies the separation axiom $T_i$, then so does the space
$D\cup_\pi M$.
\end{prop}

%\begin{proof} The space $\alpha D$, being compact, is a $T_{3\frac12}$-space. If $M$ is a $T_i$-space for some $i\le 3\frac12$, then the product $M\times \alpha D$ is a $T_i$-space by Theorem 2.3.11 of \cite{En}, and the subspace $D\cup_\pi M$ of $M\times \alpha D$ is a $T_i$-space by Theorem 2.1.6 of \cite{En}.
%\end{proof}

\begin{prop}\label{p1.2}
If $M$ is (separable) metrizable and $D$ is countable, then the
space $D\cup_\pi M$ is (separable) metrizable too.
\end{prop}

\begin{prop}\label{p1.3}
If the space $M$ is compact, then so is the space $D\cup_\pi M$.
\end{prop}

We recall that a topological space $X$ is {\em countably compact}
if each countable open cover of $X$ has a finite subcover. This is equivalent to saying that the space $X$ contains no infinite closed discrete subspace.

\begin{prop}\label{p1.4}
If some power $M^\kappa$ of the space $M$ is countably compact,
then the power $(D\cup_\pi M)^\kappa$ is countably compact too.
\end{prop}

\begin{proof}
Since $D\cup_\pi M$ is a closed subspace of $\alpha D\times M$,
the power $(D\cup_\pi M)^\kappa$ is a closed subspace of $(\alpha
D\times M)^\kappa$. So, it suffices to check that the latter space
is countably compact. Since the product of a countably compact
space and a compact space is countably compact \cite[3.10.14]{En},
the product $M^\kappa\times (\alpha D)^\kappa$ is countably compact and so
is its topological copy $(\alpha D\times M)^\kappa$.
\end{proof}

If the space $M$ is Tychonoff, then $D\cup_\pi M$ is a subspace of
the compact Hausdorff space $D\cup_\pi \beta M$ where $\beta M$ is
the Stone-\v Cech compactification of $M$. Assuming that $M$ is
countably compact at $\pi(D)$ we shall show that $D\cup_\pi \beta
M$ coincides with the Stone-\v Cech compactification of $D\cup_\pi
M$.
\smallskip

We shall say that a topological space $X$ is {\em countably
compact} at a subset $A\subset X$ if each infinite subset
$B\subset A$ has an accumulation point $x$ in $X$. The latter
means that each neighborhood $O(x)$ of $x$ contains infinitely
many points of the set $B$.

\begin{prop}\label{p1.5}
If the space $M$ is Tychonoff and is countably compact at the
subset $\pi(D)$, then $D\cup_\pi\beta M$ is the Stone-\v Cech
compactification of $D\cup_\pi M$.
\end{prop}

\begin{proof} By Proposition~\ref{p1.5}, the space $D\cup_\pi M$ is Tychonoff and hence has the Stone-\v Cech
compactification $\beta(D\cup_\pi M)$. Since the space $M$ is a
retract of $D\cup_\pi M$, the compactification
$\beta M$ is a retract of $\beta(D\cup_\pi M)$. Let $\beta
i:\beta(D\cup_\pi M)\to D\cup_\pi\beta M$ be the Stone-\v Cech
extension of the identity inclusion $i:D\cup_\pi M\to
D\cup_\pi\beta M$. We claim that $\beta i$ is a homeomorphism.

First we show that the subset $D\cup\beta M\subset\beta(D\cup_\pi
M)$ is compact. Indeed, given an open cover $\U$ of $D\cup\beta M$
we can find a finite subcover $\V\subset\U$ of $\beta M$ and then
consider the set $D'=D\setminus\bigcup\V$. We claim that this set
$D'$ is finite. Assuming the converse and using the countable
compactness of $M$ at $\pi(D)$ we could find a point $a\in M$ such
that for every neighborhood $O(a)\subset M$ the set $\{x\in
D':\pi(x)\in O(a)\}$ is infinite. Take any open set $V\in\V$
containing the point $a$. By the definition of the topology on
$D\cup_\pi M$ there is a neighborhood $O(a)\subset M\cap V$ of $a$ in $M$ 
and a finite subset $F\subset D$ such that
$\bar\pi^{-1}(O(a))\setminus F\subset V$.

Then the set $\{x\in D':\pi(x)\in O(a)\}$ lies in $F$ and hence is
finite, which is a contradiction. Hence the set $D'$ is finite and
we can find a finite subfamily $\W\subset\U$ with
$D'\subset\bigcup\W$. Then $\V\cup\W\subset\U$ is a finite subcover
of $D\cup\beta M$. Now we see that the subset $D\cup\beta M$,
being compact and dense in $\beta(D\cup_\pi M)$, coincides with
$\beta(D\cup_\pi M)$. It follows that the continuous map $\beta
i=\beta i|D\cup\beta M$ is bijective and hence is a homeomorphism.
\end{proof}

Following A.V.~Arkhangel'skii \cite[III.\S4]{Arh}, we say that a topological space $X$ is {\em countably pracompact} if $X$ is countably compact at a dense subset of $X$. It is clear that each countably compact space is countably pracompact.

\begin{prop}\label{p1.6}
The space $D\cup_\pi M$ is countably pracompact if and only if $M$ is
countably compact at a dense subset $A\supset \pi(D)$ of $M$.
\end{prop}

\begin{proof} If the space $D\cup_\pi M$ is countably pracompact, then it is countably compact at some dense subset $A\subset D\cup_{\pi}M$. The set $A$, being dense, contains the open discrete subspace $D$ of $D\cup_{\pi}M$. The continuity of the retraction $\bar \pi:D\cup_\pi M\to M$ implies that the space $M$ is countably compact at the dense subset $\bar\pi(A)\supset\pi(D)$ of $M$, so $M$ is countably pracompact.

Now assume conversely that the space $M$ is countably compact at a dense subset $A\supset\pi(D)$. We claim that $D\cup_\pi M$ is countably compact at the dense subset $D\cup A$. We need to check that each infinite subset $B\subset D\cup A$ has a cluster point in $D\cup_\pi M$. If $B\cap A$ is infinite, then the set $B\cap A\subset B$ has an accumulation point in $M$ because $M$ is countably compact at $A$. If $\pi(B\setminus A)$ is infinite,
then $\pi(B\setminus A)$ has an accumulation point $x$ in $M$ because of the countable compactness of $M$ at $\pi(D)\subset A$. By the definition of the topology on $D\cup_\pi M$, the point $x$ is an accumulation point of the set $B\setminus A$. It remains to consider the case when the sets $A\cap B$ and $\pi(B\setminus A)$ are finite. In this case for some point $c\in\pi(B\setminus A)$ the set $C=B\cap\pi^{-1}(c)$ is infinite and then $c$ is an accumulation point of the set $C\subset B$ by the definition of the topology  of $D\cup_\pi M$.
\end{proof}

A topological space $X$ is defined to be {\em pseudocompact} if
each locally finite  open cover of $X$ is finite. According to
\cite[3.10.22]{En} a Tychonoff space $X$ is pseudocompact if and
only if each continuous real-valued function on $X$ is bounded.
For each topological space we have the following implications: 
\smallskip

\centerline{countably compact $\Ra$ countably pracompact $\Ra$ pseudocompact.}
\smallskip

\begin{prop}\label{p1.7}
The space $D\cup_\pi M$ is pseudocompact if and only if $M$ is
pseudocompact and $M$ countably compact at the subset
$\pi(D)\subset M$.
\end{prop}

\begin{proof}
Assume that the space $D\cup_\pi M$ is pseudocompact. Then the space $M$ is pseudocompact, being a continuous image of the pseudocompact space $D\cup_\pi M$. 
Next, we prove that $M$ is countably compact at $\pi(D)$. Assuming
the converse, we could find a sequence $D'=\{x_n:n\in\w\}\subset
D$ such that $\pi(x_n)\ne\pi(x_m)$ for $n\ne m$ and the image
$\pi(D')$ is closed and discrete in $M$. Define an unbounded
function $f:D\cup_\pi M\to\IR$ letting
$$
f(x)=\begin{cases}n&\mbox{if $x=x_n$ for some $n\in\w$}\\
0&\mbox{otherwise,}
\end{cases}$$
and check that $f$ is continuous, which contradicts the
pseudocompactness of $D\cup_\pi M$.
\smallskip

To prove the ``if'' part, assume that the space $M$ is
pseudocompact and is countably compact at the subset $\pi(D)$. To
prove that the space $D\cup_\pi M$ is pseudocompact, fix a locally
finite open cover $\U$ of $D\cup_\pi M$ and consider the locally finite open subcover $\V=\{U\in\U:U\cap M\ne\emptyset\}$ of $M$. The pseudocompactness of
$M$ guarantees that the cover $\V$ is finite. Repeating the argument of the proof of
Proposition~\ref{p1.5}, we can check that the set $D'=D\setminus
\bigcup\V$ is finite. The local finiteness of the family $\U$ implies 
that the family $\W=\{U\in\U:U\cap D'\ne\emptyset\}$ is finite. Since $\U=\V\cup\W$, the cover $\U$ of $D\cup_\pi M$ is finite.
\end{proof}

Following \cite{BD}, we define a topological space $X$ to be {\em
openly factorizable} if every continuous map $f:X\to Y$ to a
metrizable separable space $Y$ can be written as the composition
$g\circ p$ of an open continuous map $p:X\to K$ onto a metrizable
separable space $K$ and a continuous map $g:K\to Y$.

\begin{prop}\label{p1.8}
If the set $D$ is countable and $M$ is openly factorizable, then
the space $D\cup_\pi M$ is openly factorizable too.
\end{prop}

\begin{proof}
Fix any continuous map $f:D\cup_\pi M\to Y$ to a metrizable
separable space $Y$. Since $M$ is openly factorizable, there are an
open continuous map $p:M\to K$ onto a separable metrizable space
$K$ and a continuous map $g:K\to Y$ such that $f|M=g\circ p$.

Consider the map $p\pi=p\circ\pi:D\to K$ and the corresponding
space $D\cup_{p\pi}K$ that is separable and metrizable by
Proposition~\ref{p1.2}. Let $\bar p=\id\cup p:D\cup_\pi M\to
D\cup_{p\pi}K$ be the map that is identity on $D$ and coincides
with the map $p$ on $M$. It follows from the openness of the map
$p$ that the map $\bar p$ is open (and continuous).

Now extend the map $g:K\to Y$ to a map $\bar g:D\cup_{p\pi}K\to Y$
letting $\bar g|D=f|D$. It is easy to see that $f=\bar g\circ \bar
p$. It remains to check that the map $\bar g$ is continuous. Take
any open set $U\subset Y$ and observe that $\bar g^{-1}(U)=\bar
p(f^{-1}(U))$ because $f$ is continuous and $\bar p$ is open.
\end{proof}

\section{Compact extensions of topological semigroups}

In this section we survey some known results on compact extensions
of semitopological semigroups. By a {\em semitopological
semigroup} we understand a topological space $S$ endowed with a
separately continuous semigroup operation $*:S\times S\to S$. If
the operation is jointly continuous, then $S$ is called a {\em
topological semigroup}.

Let $\C$ be a class of compact Hausdorff semitopological
semigroups. By a {\em $\C$-compactification} of a semitopological
semigroup $S$ we understand a pair $(\C(S),\eta)$ consisting of a
compact semitopological semigroup $\C(S)\in\C$ and a continuous
homomorphism $\eta:S\to\C(S)$ (called the {\em canonic
homomorphism}) such that for each continuous homomorphism $h:S\to
K$ to a semitopological semigroup $K\in\C$ there is a unique
continuous homomorphism $\bar h:\C(S)\to K$ such that $h=\bar
h\circ\eta$. It follows that any two $\C$-compactifications of $S$
are topologically isomorphic.

We shall be interested in $\C$-compactifications for the following classes of semigroups:
\begin{itemize}
\item $\WAP$ of compact semitopological semigroups;
\item $\AP$ of compact topological semigroups;
\item $\SAP$ of compact topological groups.
\end{itemize}
The corresponding $\C$-compactifications of a semitopological
semigroup $S$ will be denoted by $\WAP(S)$, $\AP(S)$, and
$\SAP(S)$. The notation came from the abbreviations for weakly
almost periodic, almost periodic, and strongly almost periodic
function rings that determine those compactifications, see
\cite[Ch.IV]{BJM}, \cite[Ch.III]{Rup}, \cite[\S21]{HS}.

The inclusions $\SAP\subset\AP\subset\WAP$ induce canonic
homomorphisms $$\eta:S\to\WAP(S)\to\AP(S)\to\SAP(S)$$ for any
semitopological semigroup $S$. It should be mentioned that the
canonic homomorphism $\eta:S\to \WAP(S)$ need not be injective.
For example, for the group $H_+[0,1]$ of orientation-preserving
homeomorphisms of the interval its WAP-compactification is a
singleton, see \cite{Megrel}.  However, for countably compact
semitopological semigroups the situation is more optimistic. The
following two results are due to E.~Reznichenko \cite{Rez}.

\begin{thm}[Reznichenko]\label{rez1}
For any Tychonoff countably compact semitopological semigroup $S$
the semigroup operation of $S$ extends to a separately continuous
semigroup operation on $\beta S$, which implies that $\beta S$
coincides with the $\WAP$-compactification of $S$.
\end{thm}

The same conclusion holds for Tychonoff pseudocompact topological
semigroups.

\begin{thm}[Reznichenko]\label{rez2}
For any Tychonoff pseudocompact topological semigroup $S$ the
semigroup operation of $S$ extends to a separately continuous
semigroup operation $\beta S$, which implies that $\beta S$
coincides with the $\WAP$-compactification of $S$.
\end{thm}

This theorem combined with the Glicksberg Theorem
\cite[3.12.20(c)]{En} on the Stone-\v Cech compactifications of
products of pseudocompact spaces, implies the following important
result, see \cite[1.3]{BD}.

\begin{thm}
For any Tychonoff topological semigroup $S$ with pseudocompact
square $S\times S$ the semigroup operation of $S$ extends to a
continuous semigroup operation on $\beta S$, which implies that
$\beta S$ coincides with the $\AP$-compactification of $S$.
\end{thm}

Another result of the same spirit involves openly factorizable
spaces with weakly Lindel\"of squares. We recall that a
topological space $X$ is {\em weakly Lindel\"of\/} if each open
cover $\U$ of $X$ contains a countable subcollection $\V\subset \U$
whose union $\bigcup\V$ is dense in $X$. The following extension
theorem is proved in \cite{BD}.

\begin{thm}\label{ap}
For any Tychonoff openly factorizable topological semigroup $S$
with weakly Lindel\"of square $S\times S$ the semigroup operation
of $S$  extends to a continuous semigroup operation on $\beta S$,
which implies that $\beta S$ is an $\AP$-compactification of $S$.
\end{thm}

The following theorem also is proved in \cite{BD}. It gives
conditions on a pseudocompact topological semigroup $S$ under which its Stone-\v Cech compactification $\beta S$
coincides with the $\SAP$-compactification \allowbreak $\SAP(S)$ of $S$.

\begin{thm}\label{sap} For a Tychonoff pseudocompact topological semigroup $S$ the Stone-\v Cech compactification $\beta S$ is a compact topological group  provided that one of the following conditions holds:
\begin{enumerate}
\item $S$ contains a totally bounded topological group as a dense subgroup;
\item $S$ contains a dense subgroup and $S\times S$ is pseudocompact.
\end{enumerate}
\end{thm}

\section{Attaching a discrete semigroup to a semitopological semigroup}

In this section we extend the construction of the space $D\cup_\pi
M$ to the category of semitopological semigroups and their
continuous homomorphisms.

Given a homomorphism $\pi:D\to M$ from a discrete semigroup $D$
into a semitopological semigroup $M$ let us extend the semigroup
operations from $(D,\cdot)$ and $(M,\cdot)$ to $D\cup_\pi M$
by letting
 $$
xy=\begin{cases}
x\cdot y&\mbox{if $x,y\in D$ or $x,y\in M$,}\\
\pi(x)\cdot y&\mbox{if $x\in D$ and $y\in M$,}\\
x\cdot \pi(y)&\mbox{if $x\in M$, $y\in D$}.
\end{cases}
$$
Endowed with the so-extended operation, the space $S=D\cup_\pi M$
becomes a semitopological semigroup containing $D$ as
a subsemigroups and $M$ as a two-sided ideal. Moreover, the map $\bar\pi=\pi\cup\id_M:D\cup_\pi M\to M$ is a continuous semigroup homomorphism.

Now we will find some conditions guaranteeing that $S=D\cup_\pi M$
is a topological semigroup.

\begin{defn}
A homomorphism $\pi:D\to M$ is called {\em finitely resolvable} if
for every $a,b\in M$ and $c\in D$ the set $\{(x,y)\in D\times
D:\pi(x)=a,\;\pi(y)=b,\; xy=c\}$ is finite.
\end{defn}

Observe that each one-to-one homomorphism is finitely resolvable.

\begin{thm}\label{t3.2}
Let $\pi:D\to M$ be a homomorphism from a discrete semigroup $D$ to a
topological semigroup $M$. For the semitopological semigroup
$S=D\cup_\pi M$ the following conditions are equivalent:
\begin{enumerate}
\item $S$ is a topological semigroup;
\item for each $c\in D$ the set $D_c=\{(x,y)\in D\times D:xy=c\}$ is closed in $S\times S$.
\end{enumerate}
If the homomorphism $\pi$ is finitely resolvable, then the conditions (1),(2) are equivalent to (3) and follow from (4):
\begin{enumerate}
\item[(3)] For each $c\in D$ the set $\pi^2(D_c)=\{(\pi(x),\pi(y)):(x,y)\in D_c\}$ is closed and discrete in $M\times M$.
\item[(4)] the subspace $\pi(D)$ is discrete in $M$ and the complement $M\setminus \pi(D)$ is a two-sided ideal in $M$.
\end{enumerate}
\end{thm}

\begin{proof}
$(1)\Ra(2)$ Assuming that $S$ is a topological semigroup, we need to check that for every $c\in D$ the set $D_c=\{(x,y)\in D\times D:xy=c\}$ is closed  in $S\times S$. Assuming the converse, we could find an accumulation
point $(a,b)\in S\times S$ for the set $D_c$. Since $D_c\subset D\times D$ is discrete, either 
$a\in M$ or $b\in M$ and hence $ab\in M$. However $ab=c$ by the continuity of the semigroup operation on $S$, which is a contradiction as $c\notin M$.
\smallskip

$(2)\Ra(1)$ Assume that for each $c\in D$ the set $D_c$ is closed
 in $S\times S$. We need to check the continuity of
the multiplication at each pair $(x,y)\in S\times S$. If $x$ or
$y$ belongs to $D$, then this follows from the continuity of left
and right shifts on $S$. So, we can assume that $x,y\in M$. Let
$z=xy$ and $O(z)\subset S$ be an open neighborhood of $z$. It
follows from the definition of the topology of $S=D\cup_\pi M$
that there are a neighborhood $U(z)\subset M$ and a finite subset
$F\subset D$ such that $\bar\pi^{-1}(U(z))\setminus F\subset
O(z)$. By the continuity of the semigroup operation on $M$ the
points $x,y$ have neighborhoods $U(x),U(y)\subset M$ such that
$U(x)\cdot U(y)\subset U(z)$.

Since the set $D_F=\bigcup_{c\in F}D_c$ is closed in
$S\times S$ and $(x,y)\notin D_F$ (because $(x,y)\in M\times M$), we
can find neighborhoods $O(x)\subset\bar\pi^{-1}(U(x))$ and
$O(y)\subset \bar\pi^{-1}(U(y))$ of the points $x,y$ such that the
set $O(x)\times O(y)$ is disjoint from the set $D_F$. In this case
$O(x)\cdot O(y)\subset S\setminus F$ and $O(x)\cdot
O(y)\subset\bar\pi^{-1}(U(x))\cdot\bar\pi^{-1}(U(y))\subset\bar\pi^{-1}(U(z))$,
which implies $O(x)\cdot O(y)\subset \bar\pi^{-1}(U(z))\setminus
F\subset O(z)$.
\smallskip

$(2)\Ra(3)$ Assume that for some $c\in D$ the set $D_c$ is closed  in $S\times S$. We shall show that its image $\pi^2(D_c)=\{(\pi(x),\pi(y)):(x,y)\in D_c\}$ is closed and discrete in $M\times M$. Assuming the converse, we could find an accumulation point $(a,b)\in M\times M$ of $\pi^2(D_c)$. We claim that $(a,b)$ is an accumulation point of the set $D_c$. Fix any neighborhoods $O(a)$ and $O(b)$ of the points $a$ and $b$ in $S$, respectively. By the definition of the topology of $D\cup_\pi M$, there are neighborhoods $U(a)$ and $U(b)$ of those points in $M$ and a finite subset $F\subset D$ such that $O(a)\supset\bar\pi^{-1}(U(a))\setminus F$ and $O(b)\supset \bar\pi^{-1}(U(b))\setminus F$. Since $(a,b)$ is an accumulation point of the set $\pi^2(D_c)$, there is a pair $(x,y)\in D_c\setminus F^2$ such that $(\pi(x),\pi(y))\in U(a)\times U(b)$. This pair $(x,y)$ belongs to the neighborhood $O(a)\times O(b)$, witnessing that $(a,b)$ is an accumulation point of the set $D_c$. Since $(a,b)\notin D_c$, the set $D_c$ is not closed in $S\times S$.
\smallskip

From now on we assume that the homomorphism $\pi$ is finitely resolvable.
\smallskip

$(3)\Ra(2)$ Assume that for some $c\in D$ the set $\pi^2(D_c)$ is closed and discrete in $M\times M$. We need to check that the set $D_c$ is closed  in $S\times S$. In the opposite case this set has an accumulation point $(a,b)$ in $S\times S$. The continuity of the retraction $\bar\pi$ implies that the pair $(\pi(a),\pi(b))$ lies in the closure of the set $\pi^2(D_c)$ and hence is an isolated point of $\pi^2(D_c)$. Then the set $W=\{(x,y)\in D_c:\pi(x)=\pi(a),\;\pi(y)=\pi(b)\}$ is infinite because $(a,b)$ is an accumulation point of $D_c$. On the other hand, the set $W$ is finite by the finite resolvability of the homomorphism $\pi$.
\smallskip

$(4)\Ra(3)$ Assume that $\pi(D)$ is discrete and $S\setminus \pi(D)$ is a two-sided ideal in $M$. Given any $c\in D$, we need to check that the subspace $\pi^2(D_c)$ is closed and discrete in $M\times M$. The space $\pi^2(D_c)$ is discrete because so is the space $\pi(D)\times\pi(D)\supset\pi^2(D_c)$. If the set $\pi^2(D_c)$ is not closed in $M\times M$, then it has an accumulation point $(a,b)\in M\times M$, which does not belongs to $\pi(D)\times \pi(D)$ as the latter space is discrete. Since $M\setminus \pi(D)$ is a two-sided ideal in $M$, $ab \notin \pi(D)$. On the other hand, by the continuity of the homomorphism $\bar\pi$, we get $\bar\pi(ab)=\pi(c)\in\pi(D_c)$ and this is a required contradiction.
\end{proof}

\begin{cor}\label{c3.3}  Let $D$ be a semigroup such that for some $c\in D$ the set $D_c=\{(x,y)\in D\times D:xy=c\}$ is infinite and let $\pi:D\to M$ be a homomorphism into a topological semigroup $M$. If $S=D\cup_\pi M$ is a topological semigroup, then $D_c$ is an open-and-closed discrete subspace of $S\times S$ and hence $S\times S$ is not pseudocompact.
\end{cor}

\section{Attaching the bicyclic semigroup to a topological semigroup}\label{s4}

In this section we study the structure of the semigroups
$D\cup_\pi M$ in the case  $D=\C(p,q)$ is the bicyclic semigroup.
The bicyclic group plays an important role in the structure theory
of semigroups, see \cite{CP}. A remarkable property of this
semigroup is that it is non-topologizable in the sense that any
Hausdorff topology turning $\C(p,q)$ into a topological semigroup
is discrete \cite{ES}.

The bicyclic semigroup $\C(p,q)$ is  generated by two element
$p,q$ and one relation $qp=1$, see \cite{CP}. It follows that each
element of $\C(p,q)$ can be uniquely written as the product
$p^nq^m$ for some $n,m\in\w$. The element $1=p^0q^0$ is a
two-sided unit for $\C(p,q)$. The product $p^mq^n\cdot p^iq^j$ of
two elements of the bicyclic semigroup $\C(p,q)$  is equal to
$p^mq^{n-i+j}$ if $n\ge i$ and to $p^{m+i-n}q^j$ if $n\le i$. The
semigroup $E_C=\{p^nq^n:n\in\w\}$ of the idempotents of $\C(p,q)$ is
isomorphic to the semigroup $\w$ of finite ordinals endowed with
the operation of maximum.

If $\pi:\C(p,q)\to H$ is any homomorphism of $\C(p,q)$ into a
group, then $\pi(1)$ is the identity element $e$ of the group $H$ and the
relation $qp=1$ implies that $\pi(q)$ and $\pi(p)$ are mutually
inverse elements of $H$, generating a cyclic subgroup of $H$, see \cite[1.32]{CP}. If the image $\pi(\C(p,q))$ is infinite, then it is easy to check that the homomorphism $\pi:\C(p,q)\to H$ is finitely resolvable.

\begin{thm}\label{t4.1}
Let $\pi:\C(p,q)\to M$ be a homomorphism of the bicyclic semigroup
into a topological semigroup $M$ such that $Z=\pi(\C(p,q))$ is a dense
infinite cyclic subgroup of $M$ and $M$ is countably compact at
$Z$. The semitopological semigroup $S=\C(p,q)\cup_\pi M$ has the
following properties:
\begin{enumerate}
\item if $M$ is countably compact, then so is $\C(p,q)\cup_\pi M$.
\item $S$ is a topological semigroup iff for every $c\in\C(p,q)$ the set
$D_c=\{(x,y)\in\C(p,q)^2:xy=c\}$ is closed in $S\times S$ iff the subsemigroup $\pi^2(D_1)=\{(\pi(q^n),\pi(p^n)):n\in\w\}$ is closed and discrete in $M\times M$.
\item $S$ is a topological semigroup provided the subgroup $Z$ is discrete in $M$ and $M\setminus Z$ is an ideal in $M$.
\item If $S$ is a topological semigroup, then the square $S\times S$  is not pseudocompact.
\item If $S$ is a topological semigroup and the space $M$ is Tychonoff, then 
\begin{enumerate}
\item $M$ is not openly factorizable,  
\item $M\times M$ is not pseudocompact, 
\item $M$ contains no dense totally bounded topological subgroup.
\end{enumerate}
\end{enumerate}
\end{thm}

\begin{proof} 1. The first item follows from Proposition~\ref{p1.4}.
\smallskip 

2. The second item will follow from Theorem~\ref{t3.2} as soon as we prove that for every $c\in\C(p,q)$ the set $\pi^2(D_c)=\{(\pi(x),\pi(y)):x,y\in\C(p,q),\; xy=c\}$ is closed and discrete in $M\times M$ provided that the subsemigroup $\pi^2(D_1)=\{(\pi(q^n),\pi(p^n)):n\in\w\}$ is closed and discrete in $M\times M$. So we assume that $\pi^2(D_1)$ is closed and discrete in $M\times M$. Taking into account that the cyclic subgroup $Z$ is dense in the topological semigroup $M$, we conclude that the semigroup $M$ is commutative and $e=\pi(1)$ is a two-sided unit of $M$. Moreover, for each $z\in Z$ the left shift $l_z:M\to M$, $l_z:x\mapsto zx$, is a homeomorphism of $M$ with inverse $l_{z^{-1}}$. This implies that for every $a,b\in Z$ the set $(a,b)\cdot \pi^2(D_1)=\{(ax,by):(x,y)\in \pi^2(D_1)\}$ is closed and discrete in $M\times M$. 

 It is easy to check that for every $c=p^iq^j\in\C(p,q)$,
$$D_c=\{(p^kq^{n},p^{i-k+n}q^j):0\le k\le i,\;n\in\w\}\cup\{(p^iq^{j-k+n},p^nq^k):0\le k\le j,\;n\in\w\}
$$ and then
$$\pi^2(D_c)=\Big(\bigcup_{0\le k\le i}(\pi(p^k),\pi(p^{i-k}q^j))\cdot\pi^2(D_1)\Big)\cup\Big(\bigcup_{0\le k\le j}(\pi(p^iq^{j-k}),\pi(q^k))\cdot\pi^2(D_1)\Big)$$
is closed and discrete in $M\times M$ (being the union of finitely many shifts of the closed discrete subspace $\pi^2(D_1)$ of $M\times M)$.
\smallskip

3. The third item follows from the implication $(4)\Ra(1)$ of Theorem~\ref{t3.2}.

4. The fourth item follows from Corollary~\ref{c3.3}.
\smallskip

5. Now assume that $S$ is a topological semigroup and the space $M$ is Tychonoff. Being countably compact at the dense subset $Z$, the space $M$ is pseudocompact.
\smallskip

5a. If the space $M$ is openly factorizable, then so is the space
$S$ according to Proposition~\ref{p1.8}. By Proposition~\ref{p1.7}, the space $S$ is pseudocompact. 
Being separable, the square $S\times S$ is weakly Lindel\"of. 
By Theorem~\ref{ap}, the Stone-\v Cech compactification $\beta S$ is a topological semigroup that contains the bicyclic semigroup $\C(p,q)$, which is forbidden by the theorem of Koch and Wallace \cite{KW}.
\smallskip

5b. Assume that the square $M\times M$ is pseudocompact. Since $M$ contains a dense cyclic subgroup $Z$, by Theorem~\ref{sap}(2), the Stone-\v Cech compactification $\beta M$ is a compact
topological group. Compact topological groups, being Dugundji
compact, are openly factorizable, which implies that $\beta M$ is
openly factorizable. By Propositions 2.3 of \cite{BD}, the open
factorizability of the Stone-\v Cech compactification $\beta M$
implies the open factorizability of the space $M$, which contradicts the preceding item.
\smallskip

5c. Assume that the semigroup $M$ contains a
dense totally bounded topological subgroup. By Theorem~\ref{sap}(1), the Stone-\v Cech compactification $\beta M$
of $M$ is a compact topological group. Further we continue as in the preceding item.
\end{proof}

\section{The structure of topological semigroups that contain bicyclic semigroups}\label{s5}

In fact, many properties of the topological semigroups
$\C(p,q)\cup_\pi M$,  established in Theorem~\ref{t4.1} hold for
any topological semigroup containing a (dense) copy of the
bicyclic semigroup $\C(p,q)$.

\begin{thm}\label{t5.1}
If a topological semigroup $S$ contains the bicyclic semigroup
$\C(p,q)$ as a dense subsemigroup, then
\begin{enumerate}
\item the complement $S\setminus\C(p,q)$ is a two-sided ideal in $S$;
\item for every $c\in\C(p,q)$ the set $D_c=\{(x,y):x,y\in\C(p,q),\;xy=c\}$
is a closed-and-open discrete subspace of $S\times S$;
\item the square $S\times S$ is not pseudocompact;
\item $\beta S$ is not openly factorizable;
\item the almost periodic compactification $\AP(S)$ of $S$ is a
compact abelian topological group and hence the canonic homomorphism
$\eta:S\to\AP(S)$ is not injective.
\end{enumerate}
\end{thm}

\begin{proof}
1. The fact that $S\setminus\C(p,q)$ is a two-sided ideal in $S$ was proved by Eberhard and Selden in \cite{ES}.
\vskip3pt

2. Given any point $c\in\C(p,q)$ we should check that the set
$D_c=\{(x,y)\in\C(p,q)^2:xy=c\}$ is an open-and-closed discrete
subspace of $S\times S$. By \cite{ES}, the topology on $\C(p,q)$
induced from $S$ is discrete. Consequently, the subspace
$\C(p,q)$, being discrete and dense in $S$, is open in $S$. Then the
square $\C(p,q)\times\C(p,q)$ is open and discrete in $S\times S$
and so is its subspace $D_c$. It remains to check that the set
$D_c$ is closed in $S\times S$. Assuming the opposite, find an
accumulation point $(a,b)\in S\times S$ of the subset $D_c$.  The
continuity of the semigroup operation implies that $ab=c$. On the
other hand, since the space $\C(p,q)\times\C(p,q)$ is discrete,
one of the points $a,b$ belong to the ideal $S\setminus \C(p,q)$ and
hence $ab\in S\setminus\C(p,q)$ cannot be equal to $c$.
\smallskip

3. The space $S\times S$ fails to be pseudocompact because it
contains the infinite closed-and-open discrete subspace $D_1=\{(x,y):x,y\in\C(p,q),\;xy=1\}=\{(q^n,p^n):n\in\w\}$.
\smallskip

4. Assuming that the Stone-\v Cech compactification $\beta S$ of
$S$ is openly factorizable, we may apply Proposition~2.3 of
\cite{BD} to conclude that $S$ is an openly factorizable
pseudocompact space. Since the space $S$ has separable (hence 
weakly Lindel\"of) square, we can apply Theorem~\ref{ap} to
conclude that $\beta S$ is a compact topological semigroup that
contains the bicyclic semigroup. But this is forbidden by the
Hildenbrandt-Koch Theorem \cite{HK}.
\smallskip

5. Let $\eta:S\to\AP(S)$ be the homomorphism of $S$ into its
almost periodic compactification. The restriction $\eta|\C(p,q)$
cannot be injective because compact topological semigroups do not
contain bicyclic semigroups. Consequently, the image
$Z=\eta(\C(p,q))$ is a cyclic subgroup of $\AP(S)$ by Corollary
1.32 of \cite{CP}. Since $\C(p,q)$ is dense in $S$, the subgroup
$Z$ is dense in $\AP(S)$.  Now Theorem~\ref{sap}(2) guarantees
that $\AP(S)$ is a compact abelian topological group.
\end{proof}

The following  theorem extends (and corrects) Theorem~2.6 of
\cite{GPR}.

\begin{thm}\label{t1}
Let $S$ be a topological semigroup containing the bicyclic
semigroup $\C(p,q)$ as a dense subsemigroup. If the space $S$ is
countably compact at the set $E_C=\{p^nq^n:n\in\w\}$ of the
idempotents of $\C(p,q)$, then
\begin{enumerate}
\item the closure $\bar E_C$ of the set $E_C$ in $S$ is compact and
has a unique non-isolated point $e$ that commutes with all elements of $S$;
\item the map $\pi:S\to S$, $\pi:x\mapsto x\cdot e=e\cdot x$,
is a continuous homomorphism that retracts $S$ onto the ideal
$M=S\setminus\C(p,q)$ having the idempotent $e$ as a two-sided unit;
\item the element $a=\pi(p)$ generates a dense cyclic subgroup $Z$ of $M$;
\item $\pi(p^nq^m)=a^{n-m}$ for all $n,m\in\w$;
\item $\lim_{n\to\infty}p^{n+k}q^n=a^k$ for every $k\in\IZ$;
\item the space $S$ is regular if and only if the space $M=S\setminus\C(p,q)$ is regular.
\item If the space $S$ is regular and countably compact at $\C(p,q)$, then  the semigroup $S$ is topologically isomorphic to $\C(p,q)\cup_\pi M$ and the subsemigroup $\{(q^ne,p^ne):n\in\w\}$ is closed and discrete in $M\times M$.
\item If the space $S$ is Tychonoff and countably compact at $\C(p,q)$, then 
the space $M$ is not openly factorizable, $M\times M$ is not pseudocompact, and the semigroup $M$ contains no dense totally bounded topological subgroup.
\end{enumerate}
\end{thm}

\begin{proof}
1. The set $E_C=\{p^nq^n:n\in\w\}$ of the idempotents of the
bicyclic semigroup $\C(p,q)$ has an accumulation point $e\in\bar
E_C$ because $S$ is countably compact at $E_C$. We claim that this
accumulation point $e$ is unique. Assume conversely that $E_C$ has
another accumulation point $e'\ne e$. Then the product $ee'$
differs from $e$ or $e'$. We lose no generality assuming that
$ee'\ne e'$. Since $S$ is Hausdorff, we can find two disjoint open
sets $O(ee')\ni ee'$ and $O'(e')\ni e'$. By the continuity of the
semigroup operation on $S$, there are two neighborhoods
$O(e)$ and $O(e')\subset O'(e')$ of the points $e$, $e'$ in $S$
such that $O(e)\cdot O(e')\subset O(ee')$. Since $e$ is an
accumulation point of the set $E_C$, we can find a number $n\in\w$
such that $p^nq^n\in O(e)$. By a similar reason, there is a number
$m\ge  n$ such that $p^mq^m\in O(e')$. Then
$$O(e')\ni p^mq^m=p^nq^n\cdot p^mq^m\in O(e)\cdot O(e')\subset O(ee'),$$which is not possible as $O'(e')$ and $O(ee')$ are disjoint.

Therefore the set $E_C$ has a unique accumulation point $e$. We
claim that the sequence $\{p^nq^n\}_{n=0}^\infty$ converges to
the point $e$. Otherwise, we would find a neighborhood $O(e)$ such
that the complement $E_C\setminus O(e)$ is infinite and hence has
an accumulation point $e'\in S\setminus O(e)$ different from $e$,
which is not possible.

This proves that the sequence $\{p^nq^n\}_{n=0}^\infty$
converges to $e$ and hence the set $\bar E_C=E_C\cup\{e\}$ is
compact and metrizable. Since the set $E=\{x\in S:xx=x\}$ of
idempotents of $S$ is closed, the accumulation point $e$ of the
set $E_C=E\cap\C(p,q)$ is an idempotent.

Next, we show that $e$ commutes with all the elements of $S$. We
start with the element $p$: $$p\cdot
e=p\cdot\lim_{k\to\infty}p^kq^k=\lim_{k\to\infty}p^{k+1}q^k=\lim_{k\to\infty}p^{k+1}q^{k+1}p=e\cdot
p.$$ By analogy we can prove that $q\cdot e=e\cdot q$. Moreover,
$$pe\cdot eq=peq=p\cdot (\lim_{k\to\infty}p^kq^k)\cdot q=\lim_{k\to\infty}pp^kq^kq=\lim_{k\to\infty}p^{k+1}q^{k+1}=e,$$
which means that the elements $pe=ep$  and $qe=eq$ are mutually inverse. It follows that the element $a=pe$ generates a cyclic subgroup $Z$ of $S$.

We claim that for every $n,m\in\w$ we have
$p^nq^m\cdot e=e\cdot p^nq^m=a^{n-m}$. Indeed, if $n\ge m$, then
$$
\begin{aligned}
p^nq^m\cdot e&=p^nq^m\cdot\lim_{k\to\infty}p^kq^k=
\lim_{k\to\infty}p^nq^mp^kq^k=\lim_{k\to\infty}p^np^{k-m}q^k=\\
&=\lim_{k\to\infty}p^{n-m}p^kq^k=p^{n-m}\lim_{k\to\infty}p^kq^k=
p^{n-m}\cdot e=(pe)^{n-m}=a^{n-m}.
\end{aligned}$$
Similarly,
$$
\begin{aligned}
e\cdot p^nq^m&=\lim_{k\to\infty}p^kq^kp^nq^m=
\lim_{k\to\infty}p^kq^{k-n}q^m=\lim_{k\to\infty}p^{n-m}p^{k-n+m}q^{k-n+m}=\\
&=p^{n-m}\lim_{k\to\infty}p^{k-n+m}q^{k-n+m}=p^{n-m}\cdot e=(pe)^{n-m}=a^{n-m}.
\end{aligned}$$
By analogy we can treat the case $n\le m$.

Therefore, $e$ commutes with all elements of the bicyclic
semigroup $\C(p,q)$. Consequently, the closed subset $\{x\in
S:xe=ex\}$ of $S$ contains the dense subset $\C(p,q)$ of $S$ and
thus coincides with $S$, which means that the
idempotent $e$ commutes with all elements of $S$.

Taking into account that the subspace $\C(p,q)$ is discrete in $S$ \cite{ES}, we conclude that the idempotent $e$, being an accumulation point of $\C(p,q)$, belongs to the complement $M=S\setminus\C(p,q)$, which is a two-sided ideal in $S$ according to Theorem~\ref{t5.1}(1). Consequently, $xe=ex\in M$ for all $x\in S$. 
\smallskip

2. It follows that the map $\pi:S\to M$, $\pi:x\mapsto
xe=ex$, is a continuous homomorphism. Let us show that $\pi(x)=x$ for every $x\in M$. Assuming the converse, find $x\in M$ with $\pi(x)\ne x$. 
It is clear that $x\ne e$. Since $S$ is Hausdorff, the points $x,e$ and $\pi(x)=xe=ex$, have  neighborhoods $O(x),O(e),O(\pi(x))\subset S$ such that $O(x)\cdot O(e)\cup O(e)\cdot O(x)\subset O(\pi(x))$ and $O(x)\cap O(\pi(x))=\emptyset$. Take any idempotent $p^kq^k\in O(e)\cap \C(p,q)$. The intersection $O(x)\cap\C(p,q)$ is infinite and hence contains a point $p^iq^j\in O(x)\cap \C(p,q)$ such that $i+j>2k$.
Then either $i>k$ or $j>k$. If $i>k$, then $$p^iq^j=p^kq^kp^iq^j\in O(x)\cap(O(e)\cdot O(x))\subset O(x)\cap O(\pi(x))=\emptyset.$$ If $j>k$, then 
$$p^iq^j=p^iq^jp^kq^k\in O(x)\cap(O(x)\cdot O(e))\subset O(x)\cap O(\pi(x))=\emptyset.$$
Both cases lead to a contradiction that completes the proof of the equality $\pi(x)=x$ for $x\in M$. This means that $\pi$ retracts $S$ onto  $M$. 
\smallskip

3. As we have already proved,
$\pi(p^nq^m)=a^{n-m}\in Z$ for every $n,m\in\w$. Since $\C(p,q)$ is dense in $S$ its image $Z=\pi(\C(p,q))$ is dense in $\pi(S)=M$. 
\smallskip

4--5. The statements (4)--(5) have been proved in the first item.
\smallskip

6. If $S$ is regular, then so is its subspace $M=S\setminus\C(p,q)$. Now assume that $M$ is regular. Given a point $x\in S$ and an open neighborhood $U$ of $x$ in $S$ we need to find a neighborhood $V$ of $x$ in $S$ such that $\overline{V}\subset U$. If the point $x$ is isolated, then we can put $V=\{x\}$. So, we assume that $x$ is non-isolated in $S$. In this case $x\in M$ (because $\C(p,q)$ is an open discrete subspace of $S$ by \cite{ES}). By the regularity of the space $M$ the point $x$ has an open neighborhood $W\subset M$ such that $\overline{W}\subset U$. The continuity of the retraction $\pi:S\to M$ implies that $V=U\cap \pi^{-1}(W)$ is an open neighborhood of $x$ in $S$. It is easy to check that this neighborhood has the required property: $\overline{V}\subset U$.
\smallskip

7. Assume that $S$ is regular and countably compact at $\C(p,q)$. We claim that the identity map $h:S\to \C(p,q)\cup_\pi M$ is a homeomorphism. The continuity of this map follows from the continuity of the map $\pi:S\to M$ and the definition of the topology of $\C(p,q)\cup_\pi M$. Since each point of $\C(p,q)$ is isolated in $\C(p,q)\cup_\pi M$, the  inverse identity map $h^{-1}:\C(p,q)\cup_\pi M\to S$ is continuous at the set $\C(p,q)$. So, it remains to check the continuity of $h^{-1}$ at a point $x\in M$. Given any neighborhood $U$ of $x$ in $S$, we need to find a neighborhood $V$ of $x$ in $\C(p,q)\cup_\pi M$ such that $V\subset U$. By the regularity of $M$, the point $x$ has an open neighborhood $W$ in $M$ such that $\overline{W}\subset U$. We claim that the set $F=\pi^{-1}(W)\setminus U$ is finite. Otherwise, by the countable compactness of $S$ at $\C(p,q)$, we can find an accumulation point $y$ of $F$. Since $F\subset S\setminus U$, the point $y$ belongs to the closed subset $S\setminus U$ of $S$. Since $\C(p,q)$ is discrete in $S$, the point $y$, being non-isolated in $S$, belongs to the complement $M=S\setminus\C(p,q)$. The continuity of the retraction $\pi:S\to M$ implies that $y=\pi(y)$ is an accumulation point of the set $\pi(F)\subset W$ and hence $y\in \overline{\pi(F)}\subset\overline{W}\subset U$, which  contradicts $y\in S\setminus U$. Thus $F$ is finite, and the set $V=\pi^{-1}(W)\setminus F$ is a required neighborhood of $x$ in $\C(p,q)\cup_\pi M$ with $\overline{V}\subset U$.

Thus $S$ is topologically isomorphic to $\C(p,q)\cup_\pi M$ and hence 
$\C(p,q)\cup_\pi M$ is a topological semigroup. By Theorem~\ref{t4.1}(2), the subsemigroup $\pi^2(D_1)=\{(q^ne,p^ne):n\in\w\}$ is closed and discrete in $M\times M$.
\smallskip

8. If $S$ is Tychonoff and countably compact at $\C(p,q)$, then $S$ is topologically isomorphic to \mbox{$\C(p,q)\cup_\pi M$} by the preceding item. Now Theorem~\ref{t4.1}(5) implies that the space $M$ is not openly factorizable, $M\times M$ is not pseudocompact, and the semigroup $M$ contains no dense totally bounded topological subgroup.
\end{proof}

\section{A countably (pra)compact semigroup that contains $\C(p,q)$}\label{s6}

In this section we shall construct a countably (pra)compact topological
semigroup containing a bicyclic
semigroup. Our main result is: 

\begin{thm}\label{t6.2}
The bicyclic subgroup is a subsemigroup of some Tychonoff countably pracompact topological semigroup.
\end{thm}

The proof of this theorem relies on four lemmas.

\begin{lem}\label{seq} A subgroup $H$ of a topological group $G$ contains  a non-trivial convergent sequence if and only if $H$ contains a non-trivial sequence that converges in $G$.
\end{lem}

\begin{proof} If a sequence $\{x_n\}_{n\in\w}\subset H$ converges to a point $x\in G\setminus H$, then $(x_{n+1}^{-1}x_n)_{n\in\w}$ is a non-trivial sequence in $H$ that converges to the neutral element $e=x^{-1}x$.
\end{proof}

The following well-known lemma can be proved by a standard argument involving binary trees.

\begin{lem}\label{l1} If a Tychonoff space $X$ is countably compact at an infinite subset $H\subset X$ that contains no non-trivial sequence that converges in $X$, then   the closure $\cl_X(A)$ of any
infinite subset $A\subset H$ has cardinality $\ge\cc$.
\end{lem}

A subset $L$ of an abelian group $G$ is called {\em linearly
independent} if for any pairwise distinct points $x_1,\dots, x_k\in L$ and any integer numbers $n_1,\dots,n_k$ the equality $n_1x_1+\dots +n_kx_k=0$ implies $n_1=\dots=n_k=0$. It is easy to see that $L\subset G$ is linearly independent if and only if for the free abelian group $\FA(L)$ generated by  $L$ the unique homomorphism $h:\FA(L)\to G$ such that
$h|L=\id_L$ is injective. For a linearly independent subset
$L\subset G$ we shall identify the free abelian group $\FA(L)$ with
the subgroup of $G$ generated by $L$.

\begin{lem}\label{l6.3} Let an Abelian torsion-free topological group $G$ is countably compact at a subgroup $H\subset G$ that contains no non-trivial convergent sequence. 
Each linearly independent
subset $L_0\subset G$ of size $|L_0|<\cc$ can be enlarged to a
linearly independent subset $L\subset G$ of size $|L|=\cc$ such
that the set $L\setminus L_0$ contains an accumulation point  of
each infinite subset $A\subset \FA(L)\cap H\subset G$.
\end{lem}

\begin{proof}
Without loss of generality, $L_0\ne\emptyset$.
Take any faithfully indexed set $X=\{x_\alpha:\alpha<\cc\}$ of
cardinality $|X|=\cc$ such that $X\cap L_0=\emptyset$ and consider
the free abelian group $\FA(L_0\cup X)$ generated by the  union
$L_0\cup X$. For every ordinal $\alpha<\cc$ let
$X_{<\alpha}=L_0\cup\{x_\beta:\beta<\alpha\}$ and
$X_{\le\alpha}=L_0\cup\{x_\beta:\beta\le\alpha\}$. So, $L_0\cup
X=X_{<\cc}$.

Denote by $\mathbf A$ the set of all countable subsets of the free
abelian group $\FA(X_{<\cc})$. Since $|\FA(X_{<\cc})|=\cc$, the set
$\mathbf A$ has size $|\mathbf A|=\cc^\w=\cc$. To each set $A\in\mathbf A$ assign the smallest ordinal $\xi(A)\le\cc$
such that $A\subset \FA(X_{<\xi(A)})$ and observe that $\xi(A)<\cc$
because $\cc$ has uncountable cofinality. It follows that
$\xi(A)=0$ if and only if $A\subset \FA(L_0)$.

We claim that there is an enumeration $\mathbf
A=\{A_\alpha:\alpha<\cc\}$ of the set $\mathbf A$ such that
$\xi(A_\alpha)\le\alpha$ for every ordinal $\alpha<\cc$. To
construct such an enumeration, first fix any enumeration $\mathbf
A=\{A'_\alpha:\alpha<\cc\}$ such that $A'_0\subset \FA(L_0)$ and
for every $A\in\mathbf A$ the set $\{\alpha<\cc:A'_\alpha=A\}$ has the 
size continuum. Next, for every $\alpha<\cc$ put
$$A_\alpha=\begin{cases}
A'_\alpha&\mbox{ if $\xi(A_\alpha')\le\alpha$}\\
A'_0&\mbox{otherwise}.
\end{cases}
$$

The identity inclusion $X_{<0}=L_0\subset G$ extends to a unique
group homomorphism $h_{<0}:\FA(X_{<0})\to G$ which is injective
because of the linear independence of $L_0$.

Inductively, for each ordinal $\alpha<\cc$ we shall construct an
injective homomorphism $h_\alpha:\FA(X_{\le\alpha})\to G$ such that
\begin{itemize}
\item $h_\alpha|\FA(X_{\le\beta})=h_\beta$ for all  $\beta<\alpha$;
\item if $h_\alpha(A_\alpha)\subset H$, then the point $\bar x_\alpha=h_\alpha(x_\alpha)\in G$ is an accumulation
point of the set $h_\alpha(A_\alpha)$.
\end{itemize}

We start with choosing a point $\bar x_\alpha$. Consider the
injective group homomorphism $h_{<\alpha}: \FA(X_{<\alpha})\to G$
such that $h_{<\alpha}| \FA(X_{\le\beta})=h_{\beta}$ for all 
$\beta<\alpha$. The image
$h_{<\alpha}(\FA(X_{<\alpha}))$ is a free abelian
subgroup of size $<\cc$ in $G$. Consider the subgroup 
$G_{<\alpha}=\big\{x\in G:\exists n>0\quad nx\in h_{<\alpha}(\FA(X_{<\alpha}))\big\}$. Since $G$ is torsion-free, $|G_{<\alpha}|\le\aleph_0\cdot |\FA(X_{<\alpha})|<\mathfrak c$.

Since the homomorphism $h_{<\alpha}: \FA(X_{<\alpha})\to G$ is
injective, the set $B_\alpha=h_{<\alpha}(A_\alpha)$ is infinite.
If $B_\alpha\subset H$, then by Lemmas~\ref{seq} and \ref{l1}, the closure $\overline{B}_\alpha$ of $B_\alpha$ in $G$ has cardinality  $|\overline{B}_\alpha|\ge\cc$.
Consequently, we can find a point $\bar
x_\alpha\in\overline{B}_\alpha\setminus G_{<\alpha}$.
If $B_\alpha\not\subset H$, then take $\bar x_\alpha$ be any point of the set 
$\overline{H}\setminus G_{<\alpha}$. Such a point $\bar x_\alpha$ exists because the closure $\overline{H}$ of $H$ in $G$ has cardinality $|\overline{H}|\ge\mathfrak c>|G_{<\alpha}|$.

The choice of the point $\bar x_\alpha\notin G_{<\alpha}$
guarantees that the injective homomorphism $h_{<\alpha}$ extends
to an injective homomorphism $h_\alpha:  \FA(X_{\le \alpha})\to G$
such that $h_{\alpha}(x_\alpha)=\bar x_\alpha$. This completes the
inductive step as well as the inductive construction.
\smallskip

Now consider the injective homomorphism $h=h_{<\cc}:\FA(X_{<\cc})\to G$ and observe that the image $L=h(X_{<\cc})$ of $X_{<\cc}=L_0\cup
X$ is a linearly independent subset of $G$. By the choice of the
homomorphism $h_{<0}$, we have $L_0=h(L_0)\subset L$.

We claim that the subgroup $\FA(L)=h(\FA(X_{<\cc})$ of $G$ generated by the set $L$ is countably compact at the subset $H\cap\FA(L)$. Take any countable 
 infinite subset $B\subset H\cap\FA(L)$ 
and consider its preimage $A=h^{-1}(B)\subset
\FA(X_{<\cc})$. It follows that $A=A_\alpha$ for some $\alpha<\cc$.
The choice of the point $\bar x_\alpha\in L\setminus L_0$
guarantees that $\bar x_\alpha$ is an accumulation point of the
set $B=h(A_\alpha)$.
\end{proof}

A (topological) semigroup $S$ is called a ({\em topological}\/) {\em
monoid} if $S$ has a two-sided unit $1$. The subgroup $H_1=\{x\in
S:\exists y\in S\;\;xy=yx=1\}$ is called {\em the maximal
subgroup} of a monoid $S$. For any subset $B$ by $\FM(B)$ we denote
the free abelian monoid generated by $M$. This is the subsemigroup
of the free abelian group $\FA(B)$ generated by the set
$B\cup\{1\}$, where $1$ is the neutral element of $\FA(B)$.

\begin{lem}\label{l6.5} Assume that a torsion-free Abelian topological group $G$ is countably compact at a dense infinite cyclic subgroup $Z\subset G$ that contains no non-trivial convergent sequence.
Then there is a Tychonoff countably
pracompact topological monoid $M$ such that
\begin{enumerate}
\item $M$ is algebraically isomorphic to the direct sum $\IZ\oplus \FM(\cc)$;
\item the maximal subgroup $H_1$ of $M$ is cyclic, discrete, and dense in $M$;
\item $M\setminus H_1$ is an ideal in $M$;
\item $M$ admits a continuous one-to-one homomorphism $h:M\to G$ such that $h(H_1)=Z$;
\item the semigroup $M$ is countably compact provided the group $G$ is countably compact and contains no non-trivial convergent sequence.
\end{enumerate}
\end{lem}

\begin{proof} Let $H=Z$ if $G$ is not countably compact and $H=G$ if $G$ is countably compact.
Let $a\in Z$ be a generator of the cyclic group $Z$. By Lemma~\ref{l6.3}, the linearly
independent set $L_0=\{a\}$ can be enlarged to a linearly
independent subset $L\subset G$ of size $|L|=\cc$ that generates the (free abelian) subgroup $\FA(L)$ in $G$ such that for
each infinite subset $A\subset H\cap \FA(L)\subset G$ the closure $\bar A$ meets the
set $L\setminus L_0$. Let $M$ be the subsemigroup of $G$ generated
by the set $\{-a,a\}\cup L$. Since each infinite subset of $Z\subset H\cap\FA(L)$ has
an accumulation point (in $L\subset M$), the space $M$ is
countably compact at the subset $Z$. 
If $G$ is countably compact, then $H=G$ and then $M$ is countably compact because each infinite subset of $M\subset H\cap \FA(L)$ has an accumulation point in $L\subset M$. 

It is clear that $M$ is a monoid whose maximal subgroup $H_1$
coincides with $Z$ and thus is dense in $M$. Also it is clear that
$M$ is algebraically isomorphic to $\IZ\oplus \FM(\cc)$.

Now we enlarge the topology $\tau$ on $M$ induced from $G$ in
order to make the maximal subgroup $Z=H_1$  discrete. It is easy
to see that the topology
$$\tau'=\{U\cup A:U\in\tau,\;A\subset Z\}$$ on $M$ has the
required property: $Z$ becomes discrete but remains dense in this
topology. It is easy to check that the space $M$ endowed with this
stronger topology remains a topological semigroup (this follows
from the fact that $M\setminus Z$ is an ideal in $M$). Moreover,
the topological space $(M,\tau')$ is Tychonoff, see
\cite[5.1.22]{En}.

It remains to check that the space $(M,\tau')$ is countably
compact at $H_1$. Take any infinite subset $A\subset H_1=Z$. By Lemma~\ref{l1},
the closure $\bar A$ of $A$ in the topology $\tau$ has size $|\bar
A|\ge\cc$ and consequently, $\bar A$ contains a point $a\notin Z$.
It follows from the definition of the topology $\tau'$ that the
point $a$ remains an accumulation point of the set $A$ in the
topology $\tau'$.

If the group $G$ is countably compact, then so is the semigroup $M$ and the preceding argument ensures that $M$ remains countably compact in the stronger topology $\tau'$. 
\end{proof}

Now we are ready to present the 

\noindent{\bf Proof of Theorem~\ref{t6.2}}. Fix an Abelian torsion-free topological group $G$ which is countably compact at a dense infinite cyclic subgroup $Z\subset G$ containing  no non-trivial convergent sequence. For $G$ we can take the Bohr compactification $b\IZ$ of the group of integers $\IZ$ and for $Z$ the image $\IZ^\sharp$ of $\IZ$ in $b\IZ$. It is well-known that the Bohr compactification $b\IZ$ is torsion-free and its subgroup $\IZ^\sharp$ contains no non-trivial convergent sequence, see \cite{vD} or \cite{GHW}. 

 By Lemmas~\ref{l6.3} and \ref{l6.5}, there is a
commutative Tychonoff topological monoid $M$ such that the maximal subgroup $H_1$ of $M$ is cyclic, discrete, and dense in
$M$, $M$ is countably compact at $H_1$, and $M\setminus H_1$ is an ideal in $M$. Let
$h:\IZ\to H_1$ be any isomorphism. Define a homomorphism
$\pi:\C(p,q)\to M$ letting $\pi(p^nq^m)=h(n-m)$ for $n,m\in\w$. By
Theorem~\ref{t4.1}(3), the semitopological semigroup
$S=\C(p,q)\cup_\pi M$ is a topological semigroup. By
Propositions~\ref{p1.1} and \ref{p1.6}, the space $S$ is Tychonoff
and countably pracompact.

Moreover, if the group $G$ is countably compact and contains no non-trivial convergent sequence, then the semigroup $M$ is countably compact according to Lemma~\ref{l6.5}(5),  and then the semigroup $S$ is countably compact by Proposition~\ref{p1.4}.
\hfill $\square$
\medskip

Let us remark that the above proof yields a bit more than required in Theorem~\ref{t6.2}, namely: 

\begin{thm}\label{t6.7} If there is a torsion-free Abelian countably compact topological group $G$ without non-trivial convergent sequences, then there exists a Tychonoff countably compact semigroup $S$ containing a bicyclic semigroup.
\end{thm}

The first example of a group $G$ with properties required in Theorem~\ref{t6.7} was constructed by M.~Tka\-chenko under the Continuum Hypothesis \cite{Tk}. Later, the
Continuum Hypothesis was weakened to Martin's Axiom for
$\sigma$-centered posets by A.~Tomita in \cite{Tom96}, for
countable posets in \cite{KTW}, and finally to the existence
of continuum many incomparable selective ultrafilters in \cite{MGT}.
Yet, the problem of the existence of a countably compact group without convergent sequences in ZFC seems to be open, see \cite{DS}.

Those consistency results combined with Theorem~\ref{t6.7} imply

\begin{cor}\label{c6.7} Martin's Axiom implies the existence of a Tychonoff countably compact topological semigroup $S$ that contains a bicyclic semigroup.
\end{cor}

\begin{rem}\label{r6.6}
By Theorem~\ref{t5.1}(5), the almost periodic compactification
$\AP(S)$ of the countably (pra)\-compact semigroup $S\supset\C(p,q)$
constructed in Theorem~\ref{t6.7} (or \ref{t6.2}) is a compact topological group.
Consequently, the canonic homomorphism $\eta:S\to\AP(S)$ is not
injective in contrast to the canonic homomorphism
$\eta:S\to\WAP(S)=\beta S$ which is a topological embedding by
Theorem~\ref{rez2}. In particular, $S$ is a countably (pra)compact topological semigroup that does not embed into a compact topological semigroup.
\end{rem}

\section{Some Open Problems}

The consistency nature of Theorem~\ref{t6.7} and Corollary~\ref{c6.7} suggests:

\begin{prob}
Is there a ZFC-example of a countably compact topological
semigroup that contains the bicyclic semigroup?
\end{prob}

Another open problem was suggested by the referee:

\begin{prob} Is there a pseudocompact topological semigroup $S$ that a contains the bicyclic semigroup as a closed subsemigroup?
\end{prob} 

Theorem~\ref{t6.2} gives an example of a countably pracompact
topological semigroup $S$ for which the canonical homomorphism
$\eta:S\to\AP(S)$ is not injective.

\begin{prob}
Is there a non-trivial countably (pra)compact topological semigroup $S$
whose almost periodic compactification $\AP(S)$ is a singleton?
\end{prob}

\section{Acknowledgment}

The authors would like to express their thanks to the referee for very careful reading the manuscripts and many valuable suggestions and to Sasha Ravsky who turned the attention of the authors to the class of countably pracompact spaces, intermediate between the classes of countably compact and pseudocompact spaces. 
%\newpage

\end{document}